\documentclass[12pt]{amsart}
\usepackage{amssymb,verbatim,enumerate,ifthen}
\usepackage{esint}
\usepackage[mathscr]{eucal}
\usepackage[utf8]{inputenc}
\usepackage{url}
\usepackage[T1]{fontenc}
\usepackage[foot]{amsaddr}
\makeatletter
\@namedef{subjclassname@2010}{%
  \textup{2010} Mathematics Subject Classification}
\makeatother

\newtheorem{thm}{Theorem}
\newtheorem{cor}{Corollary}
\newtheorem{prop}{Proposition}
\newtheorem{lem}{Lemma}

\theoremstyle{definition}

\author{Beata Der\k{e}gowska}
\email{beata.deregowska@up.krakow.pl}

\author{Pawe{\l} Pasteczka}
\email{pawel.pasteczka@up.krakow.pl}

\address{Institute of Mathematics, Pedagogical University of Krak\'ow,  Podchor\k{a}\.{z}ych 2, 30-084 Cracow, Poland}



\numberwithin{equation}{section}

\newcommand\nolabel[1]{\nonumber}

\newcommand\bff{\mathbf{f}}

\newcommand\A{\mathscr{A}}
\newcommand\R{\mathbb{R}}
\renewcommand\P{\mathbb{P}}
\newcommand\N{\mathbb{N}}

\newcommand{\QA}[1]{\mathscr{A}^{[#1]}}

\newcommand{\calP}{\mathcal{P}}
\newcommand{\PM}{\mathscr{P}}

\newcommand{\cF}{\mathcal{F}}
\newcommand{\cG}{\mathcal{G}}
\newcommand{\cFlim}{\mathcal{T}}
\newcommand{\Lo}{\mathscr{L}}
\newcommand{\Up}{\mathscr{U}}

\newcommand{\abs}[1]{\left| #1 \right| }

\newcommand{\TcF}{\mathbf{A}_{\cF}}
\newcommand{\TcG}{\mathbf{A}_{\cG}}

\DeclareMathOperator{\supp}{supp}

\DeclareMathOperator{\Var}{Var}

\DeclareMathOperator{\CM}{\mathcal{CM}}
\DeclareMathOperator{\M}{M}
\numberwithin{equation}{section}
\def\eq#1{{\rm(\ref{#1})}}
\def\Eq#1#2{\ifthenelse{\equal{#1}{*}}
  {\begin{equation*}\begin{aligned}[]#2\end{aligned}\end{equation*}}
  {\begin{equation}\begin{aligned}[]\label{#1}#2\end{aligned}\end{equation}}}

\title[Invariant means on probability space]{Quasiarithmetic-type invariant means on probability space}
\begin{document}
\maketitle
\begin{abstract}
% The aim of this paper is to generalize the notion of invariant means to infinite families of  quasiarithmetic means. 

For a family $(\mathscr{A}_x)_{x \in (0,1)}$ of integral quasiarithmetic means sattisfying certain measurability-type assumptions we search for an integral mean $K$ such that 
$K\big((\mathscr{A}_x(\mathbb{P}))_{x \in (0,1)}\big)=K(\mathbb{P})$ for every compactly supported probabilistic Borel measure $\mathbb{P}$.

Also some results concerning the uniqueness of invariant means will be given.
\end{abstract}

\section{Introduction}

For a continuous, strictly monotone function $f \colon I \to \R$ ($I$ is an interval) define a (discrete) \emph{quasiarithmetic mean} $\QA{f} \colon \bigcup_{k=1}^\infty I^k \to I$ by
\Eq{*}{
\QA{f}(x_1,\dots,x_k):=f^{-1} \Big( \frac{f(x_1)+\dots+f(x_k)}n \Big), 
}
where $k \in \N$ and $x_1,\dots,x_k\in I$. This notion was introduced in 1930s by Aumann, Knopp \cite{Kno28} and Jessen independently and then characterized by Kolmogorov \cite{Kol30}, Nagumo \cite{Nag30} and de Finetti \cite{Def31}.
For the detail concerning the early history of this family we refer the reader to the book of Hardy-Littlewood-P\'olya \cite{HarLitPol34}. From now on, a family of all continuous, strictly monotone functions on the interval $I$ will be denoted by $\CM(I)$.

It is well known that for $\pi_p\colon \R_+ \to \R$ given by $\pi_p(x):=x^p$ if $p\ne 0$ and $\pi_0(x):=\ln x,$ the quasiarithmetic mean $\QA{\pi_p}$ is a \emph{$p$-th power mean} $\PM_p$. Remarkably, the mean $\PM_1$ is the arithmetic mean.

For a vector $\bff=(f_1,\dots,f_k)$ of functions in $\CM(I)$ one can define a selfmapping $\QA{\bff} \colon I^k \to I^k$ by
\Eq{E:QAit}{
\QA{\bff}(x_1,\dots,x_k):=\big(\QA{f_1}(x_1,\dots,x_k),\dots,\QA{f_k}(x_1,\dots,x_k) \big).
}
Based on a classical results by Borwein-Borwein \cite[Theorem 8.8]{BorBor87} it is known that there exists exactly one \emph{$\QA{\bff}$-invariant mean}, that is a mean $K \colon I^k \to I$ (a function satisfying the inequality $\min(x) \le K(x) \le \max(x)$ for all $x \in I^k$) such that $K \circ \QA{\bff}=K$.
Furthermore the sequence of iterations of $\QA{\bff}$ tends to $(K,\dots,K)$ pointwise.

Inviant means in a family of quasi-arithmetic means were studied by many authors, for example Burai \cite{Bur07a}, Dar\'oczy--P\'ales \cite{DarPal02b}, J.~Jarczyk \cite{Jar07}, J.~Jarczyk and Matkowski \cite{JarMat06}. In fact invariant means were extensively studied during recent years, see for example the papers by Baj\'ak--P\'ales \cite{BajPal09b,BajPal09a,BajPal10,BajPal13}, by Dar\'oczy--P\'ales \cite{Dar05a,DarPal02c,DarPal03a}, by G{\l}azowska \cite{Gla11b,Gla11a}, by Matkowski \cite{Mat99b,Mat02b,Mat05}, by Matkowski--P\'ales \cite{MatPal15}, by Pasteczka \cite{Pas16a,Pas19a,Pas18b} and Matkowski--Pasteczka \cite{MatPas1901}. For details we refer the reader to the recent paper of J.~Jarczyk and W.~Jarczyk \cite{JarJar18}.

In (nearly) all of this paper authors refered to some counterpart of a result by Borwein-Borwein which guarantee that the invariant mean is uniquely determined. Regretfully such consideration cannot be generalized to the integral setting. 
Therefore our paper based on a recent results by Matkowski--Pasteczka \cite{MatPas1901} and Pasteczka \cite{Pas19a} for noncontinuous means.

\subsection{Integral means}
% In this section we introduce invariant mean which refers to infinite family of quasiarithmetic means. 

Hereafter $I$ stands for the arbitrary subinterval of $I$, $\mathcal{B}(I)$ and $\mathcal{L}(I)$ denote the Borel and the Lebesgue $\sigma$-algebra on $I$, respectively. Furthermore, let $\mathcal{P}(I)$ be a family of all compactly supported probabilistic on $\mathcal{B}(I)$. An (integral) mean on $I$ is a function $\M \colon \mathcal{P}(I) \to I$ such that
\Eq{*}{
\M(\P) \in [\inf \supp \P,\sup\supp \P]\text{ for all } \P\in \mathcal{P}(I). 
}
Using the notion $\gamma(\P):=[\inf \supp \P,\sup \supp \P]$ we can rewrite it briefly as $\M(\P)\in \gamma(\P)$.

Following the notion of Hardy-Littlewood-Polya \cite{HarLitPol34} for all $f \in \CM(I)$ we can define the \emph{(integral) quasiarithmetic mean} $\QA{f} \colon \mathcal{P}(I) \to I$ by
\Eq{*}{
\QA{f}(\P):=f^{-1} \big( \int f(x)\:d\P(x)\big).
}
We slightly abuse the notion of quasiarithmetic mean as $\QA{f}$ is both discrete and integral quasiarithmetic means. However it do not cause misunderstandings as they are defined of disjoint domains. Moreover for $k \in \N$ and a vector $(x_1,\dots,x_k)\in I^k$ we have
\Eq{*}{
\QA{f}\big(\tfrac{1}{k} (\delta_{x_1}+\cdots+\delta_{x_k})\big)=\QA{f}(x_1,\dots,x_n),
}
where $\delta_x$ stands for the Dirac delta. Thus this definition generalize the discrete one.
Similarly to the discrete setting we define a \emph{$p$-th power mean} by $\PM_p:=\QA{\pi_p}$.

The aim of this paper is to generalize the notion of invariant means to infinite families of integral quasiarithmetic means.

\section{Auxiliary results}

Let us first prove a simple result concerning the properties of a distance between two quasiarithmetic means.
\begin{prop}\label{prop:dfg}
Let $I \subset \R$ be a compact interval and $f,g \in \CM(I)$. Define $d_{f,g} \colon (0,|I|] \to [0,|I|]$ by
\Eq{*}{
d_{f,g}(t):=\sup_{\P \colon |\gamma(\P)| \le t}
 \abs{\QA{f}(\P)-\QA{g}(\P)}.
}
Then $d_{f,g}$ is nondecreasing and continuous.  Moreover $d_{f,g}(t)<t$ for all $t \in (0,|I|]$.
\end{prop}
\begin{proof}
Denote briefly $d \equiv d_{f,g}$. 
For $t \in (0,|I|]$ define 
\Eq{*}
{S_t:=\{(x,y,\theta) \in I\times I \times [0,1] \colon \abs{x-y}\le t\}.
}
and $m \colon I^2 \times [0,1] \to \R$ by
\Eq{*}{
m(x,y,\theta):=\abs{\QA{f}\big(\theta \delta_x+(1-\theta)\delta_y\big)-\QA{g}\big(\theta \delta_x+(1-\theta)\delta_y\big)}.
}
Then $m$ is continuous and $m(x,y,\theta)<\abs{x-y}$ unless $x=y$.

On the other hand by \cite{CarShi69} we have
\Eq{*}{
d(t)=\sup_{(x,y,\theta) \in S_t} m(x,y,\theta)=\sup_{S_t} m.
}
Since $S_t$ is compact we have $d(t)<t$ for all $ t\in(0,|I|]$.

Moreover for all $t_1 \le t_2$ we have $S_{t_1} \subseteq S_{t_2}$, thus \Eq{*}{
d(t_1)= \sup_{S_{t_1}} m \le \sup_{S_{t_2}} m =d(t_2)
}
which implies that $d$ is nondecreasing.

Now we prove that $d$ is continuous.
Fix $t_0 \in U=:U_0$ and consider a monotone sequence $(t_n)_{n=1}^\infty$, $\lim_{n\to \infty}t_n = t_0$.
Due to the monotonicity of $d$ we obtain that $(d(t_n))_{n=1}^\infty$ is convergent.

As $m$ is continuous for all $n \ge 0$ the set $S_{t_n}$ is compact, and we have 
\Eq{*}{
d(t_n)=m(s_n)\qquad \text{ for some }s_n \in S_{t_n} \subset I^2 \times [0,1],\quad n \in\{0,1,\dots\}.
}

As $I^2\times [0,1]$ is compact, there exists a subsequence $(s_{n_k})_{k=1}^\infty$ convergent to some element $\bar s$. Then $\bar s$ belongs to a topological limit of $S_{t_{n_k}}$, i.e. $\bar s \in S_{t_0}$.

Therefore
\Eq{*}{
\lim_{n \to \infty} d(t_n)=\lim_{k \to \infty} d(t_{n_k})=\lim_{k \to \infty} m(s_{n_k})=m\big(\lim_{k \to \infty} s_{n_k}\big)=m(\bar s)\le d(t_0).
}
To prove the converse inequality take a sequence $(x_n)_{n=1}^\infty$ convergent to $s_0$ such that $x_n \in S_n$ for all $n \in \N$. 
Then
\Eq{*}{
d(t_0)&=m(s_0)
=m(\lim_{n\to \infty} x_n)
=\lim_{n\to \infty}m( x_n) \le\liminf_{n\to \infty}m( s_n)
=\lim_{n\to \infty}d(t_n).
}

Therefore  $d(t_0)=\lim\limits_{n\to\infty} d(t_n)$.
\end{proof}

At the end of this section let us recall a folk result for discrete dynamical systems with a trivial attractor.
\begin{lem}\label{lem:it}
 Let $I$ be an interval with $\inf I=0$ and $d \colon I \to I$ be a continuous function such that 
 $d(x)<x$ for all $x \in I \setminus\{0\}$. Then the sequence of iterates $(d^n(x))$ converges to zero for all $x \in I$.
\end{lem}

\section{Invariance of quasiarithmetic means}
In this section we study the invariance of infinite family of quasiarithmetic means. First, we need to define a selfmapping which is a counterpart of \eq{E:QAit}. Contrary to the discrete case where such mapping is well-defined for every tuple we need some additional restrictions. 

Family $\cF:=(f_x)_{x \in[0,1]}$ of functions $f_x \colon I \to \R$ is called \emph{admissible} if 
\begin{enumerate}
 \item each $f_x$ is continuous and strictly monotone,
\item a bivariate function $I \times [0,1] \ni (t,x) \mapsto f_x(t)$ is measurable with respect to the product $\sigma$-algebra $\mathcal{B}(I)\times\mathcal{L}[0,1]$. 
 \end{enumerate}

For an admissible family $\cF:=(f_x)_{x \in[0,1]}$ and $\P\in \calP(I)$ define a measure $\TcF(\P)$ on $\R$ by
\Eq{*}{
\TcF(\P) \colon S \mapsto \abs{\{x \in [0,1] \colon \QA{f_x}(\P) \in S \} }
}
Now we are in the position to proof one of the most important results in this note.
 
\begin{lem}\label{lem:def}
Let $\cF:=(f_x)_{x \in[0,1]}$ be an admissible family and $\P \in\mathcal{P}(I)$. Then each $\QA{f_x}(\P)$ is well-defined and, moreover, $\TcF(\P) \in\mathcal{P}(I)$.
\end{lem}
\begin{proof}
Let  $h(x):=\int_If_x\:d\P$ for $x\in [0,1]$. Since $\P$ is probabilistic measure with support contained in $I$ and  $f_x$ is continuous, strictly monotone function, we have   $ h(x) \in f_x(I)$ and thus $\QA{f_x}(\P)=f_x^{-1}(h(x))$ is well defined.

Moreover by the measurability of the map $I \times [0,1] \ni (t,x) \mapsto f_x(t)$ and Fubini-Tonnelli theorem, we get that $h$ is Lebesgue measurable. Let $S\in \mathcal{B} $. Because $f_x$ is continuous, injective mapping defined on an interval, it follows that $f_x$ is homeomorphism and hence $f_x(S)$ is also a Borel measurable set. Consequently, the set 
$$
\{x \in [0,1] \colon \QA{f_x}(\P) \in S \}=h^{-1}(f_x(S))
$$
is measurable in the sense of Lebesgue. Therefore,  $\TcF(\P)$ is well-defined Borel measure on $I.$ Obviously, $\TcF(\P)(I)=1,$ which concludes the proof. 
\end{proof}

Applying the above lemma we can introduce the notion of invariance in the spirit of Matkowski. Namely for an admissible family $\cF:=(f_x \colon I \to \R)_{x \in[0,1]}$ a mean $M \colon \calP(I) \to I$ is called \emph{$\TcF$-invariant} provided
$M=M\circ \TcF$. We are going to study properties of \mbox{$\TcF$-invariant} means. Adapting some general results from \cite{Pas19a}  define the lower- and the upper-invariant mean $\Lo_\cF, \Up_\cF \colon \calP(I)\to I$ by
\Eq{*}{
\Lo_\cF(\P)&:= \lim_{n \to \infty} \Big( \inf \gamma(\TcF^n(\P))\Big), \\
\Up_\cF(\P)&:= \lim_{n \to \infty} \Big(\sup \gamma( \TcF^n(\P))\Big).
}

Now we show that these means are $\TcF$-invariant. Moreover, similarly to the discrete case, $\Lo_\cF$ and $\Up_\cF$ are the smallest and the biggest \mbox{$\TcF$-invariant} means, respectively.
\begin{thm}\label{thm:LMU}
Let $\cF:=(f_x \colon I \to \R)_{x \in[0,1]}$ be an admissible family. Then both $\Lo_\cF$ and $\Up_\cF$ are $\TcF$-invariant means. Moreover 
for every $\TcF$-invariant mean $M\colon \mathcal{P}(I) \to I$ the inequality $\Lo_\cF\le M\le \Up_\cF$ holds. 
\end{thm}

\begin{proof}
Take $\P \in \calP(I)$ arbitrarily. By virtue of Lemma~\ref{lem:def} we obtain that $\TcF^{n}(\P) \in \calP(I)$ for all $n \in \N$.

Moreover as $\QA{f_x}(\TcF^{n}(\P)) \in \gamma(\TcF^{n}(\P))$ for all $x \in [0,1]$ we obtain $\gamma(\TcF^{n+1}(\P)) \subseteq \gamma(\TcF^{n}(\P))$. In particular for every $\P \in \calP(I)$ we have
\Eq{*}{
\Lo_\cF(\P)&= \lim_{n \to \infty} \Big( \inf  \gamma(\TcF^n(\P))\Big)\subseteq \gamma(\TcF^0(\P))=\gamma(\P),
}
which proves that $\Lo_\cF$ is a mean. Moreover 
\Eq{*}{
\Lo_\cF(\P)&= \lim_{n \to \infty} \Big( \inf  \gamma(\TcF^n(\P))\Big)=\lim_{n \to \infty} \Big( \inf  \gamma(\TcF^{n+1}(\P))\Big)\\
&=\lim_{n \to \infty} \Big( \inf  \gamma\big(\TcF^{n}\big(\TcF(\P)\big)\big)\Big)=\Lo_\cF\circ\TcF(\P),
}
which shows that $\Lo_\cF$ is $\TcF$-invariant. Similarly $\Up_\cF$ is an $\TcF$-invariant mean.

Now let $M \colon \mathcal{P}(I) \to I$ be an arbitrary $\TcF$-invariant mean. Then, applying the definition of $\TcF$-invariance iteratively, we obtain
\Eq{*}{
M(\P)=M \circ \TcF^n(\P) \text{ for all }\P \in \mathcal{P}(I)\text{ and }n \in \N.
}

By mean property it follows that for all $\P \in \mathcal{P}(I)$ we have 
\Eq{*}{
M(\P)\in \gamma(\TcF^n(\P)) \qquad (n \in \N)
}
and therefore, as $\gamma(\TcF^n(\P))\subseteq \gamma(\TcF^{n-1}(\P))$, we obtain 
\Eq{*}{
 M(\P)\in \bigcap_{n=1}^\infty \gamma(\TcF^n(\P))=\big[\Lo_\cF(\P),\Up_\cF(\P)\big].
}
The latter inequality can be rewritten as $\Lo_\cF \le M \le \Up_\cF$. 
\end{proof}

\subsection{Conjugacy of means}
Following the idea contained in Bullen \cite{Bul03} and Chudziak-P\'ales-Pasteczka \cite{ChuPalPas19}, let us introduce the notion of conjugancy of means. For a continuous and strictly monotone function $u \colon J \to I$ and a mean $M\colon \mathcal{P}(I) \to I$ define a the conjugancy $M^{[u]}\colon \mathcal{P}(J) \to J$ by
\Eq{*}{
M^{[u]}(\P)=u^{-1} \Big( M\big( u(x)\:d\P(x)\big)\Big)\ .
}

It is easy to see that $(M^{[u]})^{[u^{-1}]}=M$. Moreover for every $f \in \CM(I)$  the quasiarithmetic mean $\QA{f}$ is a $f$-conjugant of the arithmetic mean (which coincides with the expected value). 

The following lemma is easy to see
\begin{lem}\label{lem:conj}
 Let $\cF:=(f_x \colon I \to \R)_{x \in[0,1]}$ be an admissible family, $u \colon J \to I$,  and $\cG:=(g_x = f_x \circ u)_{x \in[0,1]}$. Then $M\colon \mathcal{P}(I) \to I$ is a $\TcF$-invariant mean  if and only if  $M^{[u]}$ is a $\TcG$-invariant mean.
\end{lem}

\subsection{Uniqueness of invariant means}
In what follows we show few sufficient condition in order to guarantee the uniqueness of $\TcF$-invariant mean. First observe that Theorem~\ref{thm:LMU} has the following corollary
\begin{cor}\label{cor:LMU}
Let $\cF:=(f_x \colon I \to \R)_{x \in[0,1]}$ be an admissible family. Then $\Lo_\cF=\Up_\cF$ if and only is there exists exactly one $\TcF$-invariant mean.
\end{cor}

The main disadvantage of this result is that it is very difficult to verify this condition in practice. In the next result we show that whenever $\cF$ is bounded from one side then the invariant mean is uniquely determined in a weak sense.

\begin{thm}
 Let $\cF:=(f_x \colon I \to \R)_{x \in[0,1]}$ be an admissible, upper (lower) bounded family. Then 
 there exists a (uniquely determined) $\TcF$-invariant mean $K_\cF \colon \mathcal{P}(I) \to I$ such that 
 \Eq{E:m}{
\lim_{n \to \infty}\QA{k}\circ \TcF^n(\P)=K_\cF (\P)\qquad \text{ for all }k \in \CM(I)\text{ and }\P\in \calP(I). 
}
% Moreover if $\cF$ is both lower and upper bounded then $K_\cF$ is a unique $\TcF$-invariant mean.
\end{thm}
\begin{proof}
Assume that $\QA{f_x} \le \QA{u}$ for some $u \colon I \to \R$ and define 
\Eq{*}{
\cG:=(g_x:=f_x \circ u^{-1})_{x\in[0,1]}.
}
As $\QA{f_x} \le \QA{u}$  we get $\QA{g_x}\le \A$ for all $x \in [0,1]$.

% First we prove that $\TcG$-invariant mean is uniquely determined. 
Take $\P_0 \in \calP(u(I))$ arbitrarily and let $\P_n:=A_\cG(\P_{n-1})$ for all $n \in \N_+$. Then we know that 
\Eq{*}{
\PM_1(\P_{n+1})\le\PM_2(\P_{n+1})\le\sup \gamma(\P_{n+1}) \le A(\P_n)=\PM_1(\P_n).
}
This implies that all intervals $[\PM_1(\P_{n}),\PM_2(\P_{n})]$ are disjoint.
In particular
\Eq{ElimP2P1}{
\lim_{n \to \infty} \PM_2(\P_{n})=\lim_{n \to \infty}\PM_1(\P_n)=:m(\P).
}
Thus
% , as $\gamma(\P_0)$ is compact and $\gamma(\P_{n+1})\subseteq \gamma(\P_n)$ for all $n \in \N$
we obtain
\Eq{ElimVar}{
\lim_{n \to \infty} \Var(\P_n)=\lim_{n \to \infty} (\PM_2(\P_{n}))^2-(\PM_1(\P_n))^2=0.
}

In view of Chebyshev's inequality we have 
\Eq{*}{
\Pr\big(\abs{\P_n-m(\P)}\ge\xi\big)&\le\Pr\big(\abs{\P_n-\PM_1(\P_n)}\ge\xi-\abs{\PM_1(\P_n)-m(\P)}\big)\\
&\le \frac{\Var(\P_n)}{(\xi-\abs{\PM_1(\P_n)-m(\P)})^2}\text{ for all }\xi>0.
}
Whence in view of \eq{ElimP2P1} and \eq{ElimVar} we obtain
\Eq{*}{
\lim_{n \to \infty} \Pr\big(\abs{\P_n-m(\P)}\ge\xi\big) =0 \text{ for all }\xi>0
}
which shows that $\P_n \to \delta_{m(\P)}$ in a (Lebesgue) measure. As each $\P_n$ is compactly supported we obtain that 
\Eq{*}{
\lim_{n \to \infty}\QA{k}(\P_n)=m(\P)\qquad \text{ for all }k \in \CM(u(I)).
}
Consequently, as $\QA{f_x}=\QA{g_x\circ u}=(\QA{g_x})^{[u]}$ for all $x \in [0,1]$ we have
\Eq{*}{
\lim_{n \to \infty}\QA{k}(\TcF^n(\P))=m^{[u]}(\P)\qquad \text{ for all }k \in \CM(I)\text{ and }\P \in \calP(I)
}
which yields \eq{E:m} with $K_\cF:=m^{[u]}$.
\end{proof}

Now we show a result in a case when the family $\cF$ satisfy some sort of boundedness. It is important to emphasize that even a finite family of quasi-arithmetic means can be unbounded (in the family of quasiarithmetic means with a pointwise ordering), see \cite{Pas19b} for details.

\begin{thm} 
 Let $\cF:=(f_x \colon I \to \R)_{x \in[0,1]}$ be an admissible family and $\cFlim$ be a finite subset of $\CM(I)$. Assume that for every $x \in [0,1]$ there exists $l_x,u_x \in \cFlim$ such that $\QA{l_x}\le\QA{f_x}\le \QA{u_x}$. Then there exists a uniquely determined $\TcF$-invariant mean $K_\cF \colon \mathcal{P}(I) \to I$.
\end{thm}

\begin{proof}
Define $d\colon(0,|I|] \to [0,|I|)$ by
\Eq{*}{
d(t):=\max_{l,u \in \cFlim} \sup_{\P \colon |\gamma(\P)| \le t} \abs{\QA{l}(\P)-\QA{u}(\P)} =\max_{l,u \in \cFlim} d_{l,u}(x). 
}
Then by Proposition~\ref{prop:dfg} we obtain that $d$ is continuous and $d(x)<x$ for all $x\in(0,|I|]$. Therefore by Lemma~\ref{lem:it} we obtain that the sequence of iterations $(d^n)_{n=1}^\infty$ tends to zero pointwise.

On the other hand for every mean $\P \in \mathcal{P}(I)$ and  $x \in [0,1]$ we obtain
\Eq{*}{
\min_{l \in \cFlim} \QA{l}(\P) \le \QA{f_x}(\P) \le \max_{u \in \cFlim} \QA{u}(\P)
}
Therefore
\Eq{*}{
\sup_{x \in[0,1]} \QA{f_x}(\P) -\inf_{x \in [0,1]} \QA{f_x}(\P) \le \max_{u \in \cFlim} \QA{u}(\P) - \min_{l \in \cFlim} \QA{l}(\P) \le d(|\gamma(\P)|).
}
Thus $|\gamma(\TcF(\P))| \le d(|\gamma(\P)|)$ for every $\P \in \calP(I)$. Therefore
\Eq{*}{
\Up_\cF(\P)-\Lo_\cF(\P)=\lim_{n \to \infty} |\gamma( \TcF^n(\P))| \le \lim_{n \to \infty} d^n(|\gamma(\P)|)=0,
}
which proves $\Up_\cF(\P)=\Lo_\cF(\P)$. As $\P$ was taken arbitrarily we obtain $K_\cF:=\Up_\cF=\Lo_\cF$, which by Corollary~\ref{cor:LMU} implies that $\TcF$-invariant mean is uniquely determined.
\end{proof}

Applying this theorem we can easily show the finite counterpart of this result
\begin{cor}
 Let $\cF:=(f_x \colon I \to \R)_{x \in[0,1]}$ be an admissible family which contains finitely many distinct functions. Then there exists a uniquely determined $\TcF$-invariant mean $K_\cF \colon \mathcal{P}(I) \to I$.
\end{cor}

\end{document}